\documentclass[12pt,reqno]{amsart}
\usepackage[a4paper, total={6in, 8in}]{geometry}
\usepackage{amsthm}
\usepackage{amsmath}
\usepackage{etoolbox}
\usepackage{amsfonts}
\usepackage[shortlabels]{enumitem}
\usepackage{amssymb}
\usepackage[all,cmtip]{xy}
\newtheorem{thm}{Theorem}[section]

\newtheorem{cor}[thm]{Corollary}

\newtheorem{remark}[thm]{Remark}

\newtheorem{ex}[thm]{Example}

\newcommand{\norm}[1]{\left\|#1\right\|}

\newcommand{\bb}[1]{\mathbb{#1}}
\newcommand{\cl}[1]{\mathcal{#1}}

\newcounter{egcounter}
\begin{document}
\title[GKZ theorem for polynomials]{Multiplicativity of linear functionals on function spaces on an open unit disc}
\author{Jaikishan}
\address{Department Of Mathematics\\
      SNIoE (deemed to be university)\\
   School of Natural Sciences\\
  Gautam Budh Nagar - 203207\\
         Uttar Pradesh, India}
\email{jk301@snu.edu.in}

\author{Sneh Lata}
\address{Department Of Mathematics\\
       SNIoE (deemed to be universty)\\
        School of Natural Sciences,\\
         Gautam Budh Nagar - 203207\\
         Uttar Pradesh, India}
\email{sneh.lata@snu.edu.in}

\author{ Dinesh Singh}
\address{Centre For Lateral Innovation, Creativity and Knowledge\\
      SGT University\\
   Gurugram 122505\\
        Haryana, India}
\email{dineshsingh1@gmail.com}

\subjclass[2010]{Primary 47B38; Secondary 30H10, 46H05, 47B37}

\keywords{Gleason-Kahane-$\dot{\text{Z}}$elazko theorem, multiplicative linear functionals, Hardy spaces, weighted Hardy spaces, strictly cyclic weighted shift.}

\begin{abstract}
This paper presents a fairly general version of the well-known Gleason-Kahane-$\dot{\text{Z}}$elazko (GKZ) theorem in the spirit of a GKZ type theorem obtained recently by Mashreghi and Ransford for Hardy spaces. In effect, we characterize a class of linear functionals as point evaluations on the vector space of all complex polynomials $\cl P$. We do not make any topological assumptions on $\cl P$. We then apply this characterization to present a version of the GKZ theorem for a vast class of topological spaces of complex-valued functions including the Hardy, Bergman, Dirichlet, and many more well-known function spaces. We obtain this result under the assumption of continuity of the linear functional, which we show, with the help of an example, to be a necessary condition for the desired conclusion. Lastly, we use the GKZ theorem for polynomials to obtain a version of the GKZ theorem for strictly cyclic weighted Hardy spaces.   
\end{abstract}

 \maketitle  

\section{Introduction} The very elegant and well-known Gleason-Kahane-$\dot{\text{Z}}$elazko (GKZ) theorem \cite{Gle, KZ,Zel} 
characterizes multiplicative linear functionals on Banach algebras as those linear functionals $F$ such that $F(e) = 1$ and 
$F(a)\ne 0$ for all invertible elements $a,$ where $e$ is the identity. (Note: The GKZ theorem makes no continuity assumptions on 
$F$). This theorem–for very evident reasons–has attracted a great deal of attention, particularly in recent times, as witnessed by the appearance of \cite{CHMR, MR, MR1, MRR, Sam}, and especially the various references in the important and useful survey article \cite{MR}.

Our interest in the GKZ Theorem stems from the work of Mashreghi and Ransford \cite{MR1} who produced a GKZ type theorem for modules over Banach algebras. They then used it to give the following version of the GKZ theorem for Hardy spaces.

\begin{thm}\label{MRH} \cite[Theorem 2.1]{MR1} Let $0<p\le \infty$ and let $F:H^p(\bb D) \to \bb C$ be a linear functional such that $F(1)=1$ and $F(g)\ne 0$ for every outer function $g$ in $H^p(\bb D).$ Then there exists a scalar 
$w\in \bb D$ such that $F(f)=f(w)$ for all $f$ in $H^p(\bb D.)$
\end{thm}

This is a true analogue of the GKZ theorem on Hardy spaces since any function $g$ in a Hardy space for which there exists an $h$ in the space such that $gh$ = 1 has to be outer and a point evaluation on a Hardy space is multiplicative whenever the multiplication makes sense. In the same paper \cite{MR1} the authors also proved the following GKZ type theorem for a wide variety of Banach spaces.

\begin{thm}\label{MRG} \cite[Theorem 3.1]{MR1} Let 
$\cl X$ be a Banach space consisting of analytic functions on the open unit disc $\bb D$ satisfying the following properties:
\begin{enumerate}[(i)]
\item for each $w\in \bb D,$ the evaluation map at $w$ is continuous on $\cl X;$
\item $\cl X$ contains all polynomials and polynomials are dense in $\cl X;$
\item multiplication with the function $z$ is well-defined on $\cl X.$ 
\end{enumerate}
Furthermore, let $\cl Y$ be a subset of $\cl X$ with the following properties:
\begin{enumerate}[(i)]
\item if $g\in \cl X$ with $0<\inf_{z\in \bb D}|g(z)|\le \sup_{z\in \bb D}|g(z)|<\infty,$ then $g\in \cl Y;$
\item $\cl Y$ contains all linear polynomials $z-\alpha$ with $|\alpha|=1.$
\end{enumerate}
If $F$ is a continuous linear functional on $\cl X$ such that $F(1)=1$ and $F(g)\ne 0$ for each $g\in \cl Y,$ then there exists a scalar $w\in \bb D$ such that $F(f)=f(w)$ for all $f\in \cl X.$
\end{thm}

Both these results motivated us to investigate the possibility of obtaining a version of GKZ theorem for function spaces by knowing the behavior of the linear functional on a smaller and simpler class of functions in them. 

In this paper, we prove a far reaching and very general version of the GKZ theorem, as in Theorem A stated below, that characterizes all multiplicative linear functionals amongst all linear functionals on the algebraic vector space $\cl P$ of all complex polynomials. In effect, Theorem A characterizes multiplicative linear functionals amongst all linear functionals 
$F$ on $\cl P$ as those that satisfy $F(1) = 1$ and $F((z-\alpha)^n) \ne 0$ for all $n = 0, 1,2,3,\dots$ and for all $\alpha$ outside some fixed open disk with centre $0$. We establish that such functionals have to necessarily be point evaluations on $\cl P$. Notice that polynomials of the form $(z-\alpha)^n$ for $n\ge 1$ and scalars $\alpha$ with $|\alpha|\ge 1$ are in fact outer functions in the Hardy space $H^p(\bb D).$ See Theorem A below for a precise version.  

\vspace{.2 cm}

\noindent {\bf Theorem A.} {\em Let $\cl P$ denote the vector space of all complex polynomials and let $F: \cl P\to \bb{C}$ be a linear functional. Suppose $F(1)=1$ and there exists a positive real number $r$ such that 
$F((z-\alpha)^n)\ne 0$ for all $n\ge 1$ and for all complex numbers $\alpha $ with $|\alpha|\ge r.$ Then 
there exists a complex number $w$ with $|w|<r$ such that $F(p)=p(w)$ for all polynomials $p\in \cl P.$}

\vspace{.2 cm}

We want to point out the following notable facts about our theorem:
\begin{enumerate}[(i)]
\item Our Theorem A tells us that the multiplicative linear functionals on $\cl P$ are characterized by being non-zero on a very special and limited class of powers of linear polynomials that mimic invertibility in some sense.
\item Our proof is very elementary and relies on simple algebraic properties of polynomials and is quite like the comparatively recent and elementary proof of the original GKZ theorem by Roitman and Sternfeld as given in \cite{RS}.
\end{enumerate}

As an application, we obtain the following GKZ theorem for topological spaces. Although different, it is in the same spirit 
as Theorem \ref{MRG} (\cite[Theorem 3.1]{MR1}). Let $\cl F(r,\bb C)$ denote the set of all complex-valued functions on the open disc with center at $0$ and radius $r.$ Clearly it is a vector space over $\bb C$ with respect to pointwise addition and scalar multiplication. 

\vspace{.2 cm}

\noindent{\bf Theorem B.} {\em Let $\cl X$ be a vector subspace of $\mathcal{F}(r,\mathbb{C})$. Suppose $\cl X$ also has a topology (not necessarily compatible with the linear structure on $\cl X$) such that: 
\begin{enumerate}[(i)]
\item for each $w$ with $|w|<r$, the point evaluation map $f \mapsto f(w)$ is continuous on $\cl X$;
\item $\cl X$ contains all the polynomials and they are dense in $\cl X.$   
\end{enumerate}
Let $F: \cl X \to \mathbb{C}$ be a continuous linear functional such that $F(1)=1$ and $F((z-\alpha)^n)\ne 0$ whenever $n\ge 1$ and $\alpha$ is a scalar with $|\alpha|\ge r$. Then there exists a scalar $w$ with $|w|<r$ such that 
\begin{equation*}
F(f)=f(w) 
\end{equation*}
for every $f\in \cl X.$}

\vspace{.2 cm} 

Note that all the function spaces mentioned in \cite[Page 1018] {MR1} satisfy the hypotheses of Theorem B. Moreover, since we do not require the space $\cl X$ to be a Banach space, we can further include a few more well-known spaces for which the theorem holds such as the Hardy spaces $H^p(\bb D) \ (0 < p < 1) $ and the Bergamn spaces $A^p(\bb D) \ (0 < p < 1).$ In addition to these, we also cover the weighted Hardy spaces $H^2(\beta)$ with $\liminf(\beta_n)^{1/n}>0.$ The following are some more salient features of our Theorem B. 

\begin{enumerate}[(i)]

\item  For $r=1$ in Theorem B, the polynomials $(z-\alpha)^n$ for $n\ge 1$ and scalars $\alpha$ with $|\alpha|\ge 1$ are indeed outer functions in the Hardy spaces $H^p(\bb D)$. So, our Theorem B (with $\cl X=H^p(\bb D))$, just like Theorem \ref{MRH}, is also an analogue of a GKZ theorem for Hardy spaces.  Although our version seems more elegant as our hypotheses involve a smaller and special class of outer functions, it may be thought otherwise as we assume continuity of the linear functional, which follows automatically in Theorem \ref{MRH}. However, in Example \ref{cont}, we establish the necessity of continuity for our result.

%
%
\item Our proof of Theorem B is elementary and simple, and unlike the proofs of Theorem \ref{MRH} and \ref{MRG}, as given in \cite{MR1}, it does not depend on the classical GKZ theorem.

\item We do not require the space $\cl X$ to be a Banach space. We need it to be a vector space with a topology but do not require the topology to be compatible with the linear structure. As a result, as noted above, we cover more spaces as compared to Theorem \ref{MRG}.
\end{enumerate}

We also use Theorem A to prove a GKZ theorem for the weighted Hardy spaces on which the operator of multiplication with $z$ is strictly cyclic. For the reader's convenience, we recall the definition of a weighted Hardy space here; however, we defer the definition of a strictly cyclic operator until Section \ref{cyclic}. 

Let $\{\beta_n\}_{n=0}^\infty$ be a sequence of positive numbers with $\beta_0=1.$ Then the weighted Hardy space $H^2(\beta)$ 
consists of formal power series $f(z)=\sum_{n=0}^\infty a_nz^n$ such that 
$\sum_{n=0}^\infty |a_n|^2\beta_n^2$ is finite. The norm on $H^2(\beta)$ is given by $||f||^2_\beta=\sum_{n=0}^2|a_n|^2\beta_n^2.$ The equality 
$f(z)=\sum_{n=0}^\infty a_nz^n$ simply means that $f$ denotes this formal power series; it does not mean convergence of the series at any complex number 
$z.$ The sequence $\{z^n\}_{n=0}^\infty$ is an orthogonal basis for $H^2(\beta)$ with $||z^n||_\beta=\beta_n$ for each $n.$ 
If $sup_{n}{\beta_{n+1}/{\beta_n}}$ is finite, then $M_z(z^n)=z^{n+1}$ gives rise to a well-defined bounded operator on $H^2(\beta).$ Clearly, the converse is also true. In this article, we shall assume $M_z$ to be a bounded operator on $H^2(\beta).$

\vspace{.2 cm}  

\noindent {\bf Theorem C.} {\em Let $M_z,$ the operator of multiplication with $z,$ be strictly cyclic on the weighted Hardy space $H^2(\beta)$. Let $F:H^2(\beta)\to \bb C$ be a linear functional such that $F(1)=1$ and $F(f)\ne 0$ for every invertible $f\in H^2(\beta).$ Then there exists a scalar $w$ with $|w|\le r(M_z)$ such that 
$$
F(f)=f(w)
$$
for all $f\in H^2(\beta).$ }

\vspace{.2 cm} 
 
Of course, some functions in $H^2(\beta)$ are entire functions, but in general, the disc 
$\{w\in \bb C: |w|< \liminf\beta_n^{1/n}\}$ is the largest disc in which every formal power series in $H^2(\beta)$ converges. At this point, we want to remind that all weighted Hardy spaces with $\liminf\beta_n^{1/n}>0$ are already covered by Theorem B. However, there are weighted Hardy spaces  for which $\liminf\beta_n^{1/n}=0$ and on which the operator of multiplication with $z$ is strictly cyclic. We shall give one such example in Section \ref{cyclic}. All such spaces clearly lie outside the purview of Theorem B; so by means of Theorem C, we  establish a GKZ theorem for a whole new class of weighted Hardy spaces. 

 \section{A Brief Preview}
The organization of the paper is as follows. We prove Theorem A in Section \ref{poly} and use it to prove Theorem B in Section \ref{metric} and Theorem C in Section \ref{cyclic}. In Example \ref{cont}, 
we give an example that validates our assumption of continuity of the linear functional in Theorem B. Indeed, we construct an unbounded linear functional on the Hardy space $H^2(\bb D)$ such that it is 1 at the constant function 1 and it is non-zero on polynomials $(z-\alpha)^n$ for every $n\ge 1$ and each $\alpha$ with $|\alpha|\ge 1.$ This functional being unbounded can never be a point evaluation. In Corollary \ref{wcomp0}, we obtain a characterization of weighted composition operators for our setting.

\section{Proof of Theorem A: The GKZ theorem for polynomials}\label{poly}
In this Section, we shall prove Theorem A. Our proof is interesting in its own right and also because we 
extend–with some ingenuity–the idea used by Roitman and Sternfeld \cite[Lemma 2]{RS} to give an 
elementary proof of the classical GKZ theorem.  

\vspace{.2 cm}

\noindent {\bf Proof of Theorem A.}
Fix any $k\ge 0$, and take $n>k$. Set  
\begin{eqnarray}
p(t)&=& F((t-z)^n)\nonumber\\
&=& {n\choose0}t^n-{n\choose1}F(z)t^{n-1}+\dots+(-1)^{n} F(z^n)\label{poly1}
\end{eqnarray}
Then $p(t)$ is a polynomial in $t$ of degree $n$. Let $\lambda_1, \lambda_2,\dots, \lambda_n$ denote the roots (counting the multiplicity) of $p.$ Then the hypothesis $p(t)=F((t-z)^n)\neq 0$ whenever 
$|t|\ge r$ implies that $|\lambda_i|<r$ for $1\le i\le n.$ Further, we can write  
\begin{equation}\label{poly2}
p(t)=\prod_{i=1}^n (t-\lambda_i).
\end{equation} 
Then by comparing the coefficients of $t^{n-1}$ from Equations (\ref{poly1}) and (\ref{poly2}), we obtain  
\begin{equation}\label{poly3}
nF(z)= \sum_{i=1}^{n}\lambda_{i},
\end{equation} 
and comparing the coefficients of $t^{n-k}$ from Equations (\ref{poly1}) and (\ref{poly2}), we obtain  
\begin{equation}\label{poly4}
\sum_{1\le i_{1}<i_{2}<...<i_{k}\le n}\lambda_{i_{1}} \lambda_{i_{2}}\dots\lambda_{i_{k}}={n\choose k}F(z^{k}).
\end{equation}
 
Thus,  
\begin{eqnarray}
&&\left(nF(z)\right)^{k}-k!{n\choose k}F(z^{k})\nonumber\\ 
&=&\left(\sum_{i=1}^{n} \lambda_i\right)^{k} - k! \sum_{1\le i_{1}<i_{2}<...<i_{k}\le n}\lambda_{i_{1}} \lambda_{i_{2}}\dots\lambda_{i_{k}}\nonumber\\
&=&\sum_{\substack{j_{1}+j_{2}+\dots+j_{n}=k \\ j_i\ge0} }\frac{k! \, (\lambda_{1}^{j_{1}}\lambda_{2}^{j_2}\dots\lambda_{n}^{j_n})}{j_{1}!\, j_{2}!\,\dots\,j_{n}!}-k!\sum_{1\le i_{1}<i_{2}<\dots<i_{k}\le n}\lambda_{i_{1}} \lambda_{i_{2}}\dots\lambda_{i_{k}}\label{poly5}.
\end{eqnarray}
Clearly, each term in the later summation of Equation (\ref{poly5}) also appears in the first summation of Equation (\ref{poly5}). Also, $|\lambda_i|<r$ for each $i.$ Therefore, if we let  
$A_{\#}$ denote the number of terms in the expression $\displaystyle\sum_{\substack{j_{1}+j_{2}+\dots+j_{n}=k \\ j_i\ge0} }\frac{k!}{j_{1}!\, j_{2}!\,\dots\,j_{n}!} (\lambda_{1}^{j_{1}}\lambda_{2}^{j_2}\dots\lambda_{n}^{j_n})$
and $B_{\#}$ denote the number of terms in the expression $\displaystyle\sum_{1\le i_{1}<i_{2}<\dots<i_{k}\le n}\lambda_{i_{1}} \lambda_{i_{2}}\dots\lambda_{i_{k}}$, then 
\begin{equation*}
\left\vert(nF(z))^{k}-k!{n\choose k}F(z^{k})\right\vert\le k!(A_\#-B_\#)r^k.
\end{equation*}
Now using standard facts from combinatorics, 
\begin{eqnarray}
A_\#-B_\# &=&{n+k-1\choose n-1}- {n\choose k} \nonumber\\
&=&\frac{(n+k-1)!}{k!(n-1)!}-\frac{n!}{k!(n-k)!}
\nonumber \\
&=&\frac{(n+k-1)(n+k-2)\dots(n+k-(k+1))!}{k!(n-1)!}-\frac{n(n-1)\dots(n-(k-1))(n-k)!}{k!(n-k)!}
\nonumber\\
&=& \frac{n^k}{k!}\left[\Big (1+\frac{k-1}{n}\Big)\Big(1+\frac{k-2}{n}\Big)\dots1-\Big(1-\frac{1}{n}\Big)\dots\Big (1-\frac{k-1}{n}\Big)\right].
\nonumber
\end{eqnarray}
Therefore, 
\begin{eqnarray}
\left\vert(nF(z))^{k}-k!{n\choose k}F(z^{k})\right\vert &\le& k!(A_\#-B_\#)r^k
\nonumber\\
&=&n^k \left[\Big (1+\frac{k-1}{n}\Big)\Big(1+\frac{k-2}{n}\Big)\dots 1
-\Big(1-\frac{1}{n}\Big)\dots\Big (1-\frac{k-1}{n}\Big)\right]r^k.
\nonumber
\end{eqnarray}
Lastly, dividing both sides by $n^k$ and letting $n \to \infty$, we conclude 

\begin{equation*}
	F(z^k)=F(z)^k.
\end{equation*}
Hence, we have established that $F(z^k)=(F(z))^k$ for all $k\ge 0.$  

Let $w=F(z).$ Then $F(p)=p(w)$ for every polynomial $p\in \cl P.$ Finally, $F(z-w)=F(z)-wF(1)=0,$ which implies that $|w|<r.$ This completes the proof.  
\qed

\section{Proof of Theorem B: The GKZ Theorem for a class of topological spaces of complex-valued functions} \label{metric}

We are now ready to prove Theorem B, our GKZ theorem for topological spaces which also have a linear structure on them. But we do not assume our underlying space to be a topological vector space. 

\vspace{.2 cm}  
 
\noindent {\bf Proof of Theorem B.} The given hypotheses yields that $F:\cl P\to \bb C$ is linear, $F(1)=1,$ and 
$F((z-\alpha)^n)\ne 0$ for all $n\ge 1$ and scalars $\alpha$ with $|\alpha|\ge r.$ Therefore, using Theorem A, we get a scalar 
$w$ with $|w|<r$ such that $F(p)=p(w)$ for all polynomials $p\in \cl P.$ 
Now, let $f \in \cl X$. Then there exists a net of polynomials $\{p_{\lambda}\}_{\lambda}$ that converge to $f$. Since point evaluation at $w$ is continuous on $\cl X$, we have 
$p_{\lambda}(w) \to f(w)$. But, $p_{\lambda}(w)=F(p_\lambda)$ and $F$ is assumed to be continuous; so, $p_\lambda(w)=F(p_\lambda)$ converges to $F(f)$. Hence, 
$F(f)=f(w).$ This completes the proof.
\qed 

\vspace{.2 cm}


The following is our example to justify the necessity of the continuity of the linear functional in Theorem B.

\begin{ex}\label{cont} Let $\mathcal{B}$ be a Hamel basis for the Hardy space $H^2(\mathbb{D})$ obtained by extending the orthonormal basis $\{z^n:n \ge 0\}$. Since $\mathcal{B}$ is an uncountable set, we can choose a countable subset $\{f_n: n\in \mathbb{N} \}$ of $\mathcal{B}$ such that $f_n\ne z^m$ for any $n$ and $m$. For a fixed $w \in \mathbb{D}$, define $F: \mathcal{B} \to \mathbb{C}$ by
$$  
F(b)=
\begin{cases}
w^n &\quad \text{if}  \quad b =z^n,\\
m\norm{b} &\quad \text{if}  \quad b=f_m,\\
0 &\quad  \quad \text{otherwise},\\
\end{cases}
$$
and extend it linearly to $H^2(\bb D).$ Then $F(1)=1$ and $F(p)=p(w)$ for every polynomial $p.$ This implies    
$F((z-\alpha)^n)=(w-\alpha)^n\ne 0$ whenever $\alpha$ is a scalar with $|\alpha|\ge 1.$ Therefore, $F$ is a linear functional on $H^2(\bb D)$ which satisfies the hypotheses of Theorem B. But $F(f_n)=n||f_n||$ for each $n$ which implies that $F$ is unbounded.  

Suppose now that $F$ is a point evaluation on $H^2(\bb D),$ and the corresponding point is $\zeta.$ Then $\zeta=F(z)=w\in \bb D.$ This implies that 
$|F(f)|=|f(w)|\le ||f||||k_w||$ for all $f$ in $H^2(\bb D),$ where $k_w$ is the kernel function at $w.$  This means $F$ is bounded, which is a contradiction. Hence, $F$ cannot be a point evaluation on $H^2(\bb D)$.
\end{ex}

As in \cite{MR1}, we also obtain the following characterization of weighted composition operators for our setting. The proof is similar to the proof of Theorem 3.2 in \cite{MR1}, we include it here for completeness. Let $Hol(\bb D)$ denote the set of all holomorphic functions on the open unit disc $\bb D$. We assume that $Hol(\bb D)$ is given its usual $Fr\grave{e}chet$-space topology.

\begin{cor}\label{wcomp0} Let $\cl X$ be a vector subspace of $\cl F(r, \bb C)$ containing all polynomials, and let it has a topology such that point evaluations on it are continuous and polynomials are dense in it.  Let $T: \cl X \to Hol(\mathbb{D})$ be a continuous linear map such that $\left(T(z-\alpha)^n\right)(\zeta)\ne0$ for all $n\ge 0,$ for all scalars $\alpha$ with 
$|\alpha|\ge r,$ and for all $\zeta\in \bb D$. Then there exist holomorphic functions $\phi: \mathbb{D} \to \mathbb{D}$ and $ \psi: \mathbb{D} \to \mathbb{C}\setminus\{0\}$ such that 
\begin{equation}\label{wcomp}
Tf=\psi.(f\circ{\phi}),
\end{equation}
 for all $f\in \cl X.$
\end{cor}

\begin{proof}
Set $\psi=T(1)$. Then, using the hypotheses, $\psi$ is holomorphic on $\bb D$ and does not assume zero at any point. Define $\phi:=(Tz)/\psi$. Then $\phi$ is clearly a holomorphic function on $\mathbb{D}$. Now we claim that $\phi$ and $\psi$ satisfy Equation (\ref{wcomp}).
	
Fix a $ \zeta \in \mathbb{D}$ and define $T_\zeta: \cl X \to \bb C$ by $T_\zeta(f) =(Tf)(\zeta)/\psi(\zeta)$. 
We can easily verify, using the assumptions on $T,$ that $T_\zeta$ is continuous, linear, and is non-zero on polynomials of the form $(z-\alpha)^n$ whenever $|\alpha|\ge r.$ Also, $T_\zeta(1)=1.$ Then, using Theorem B, we get a scalar $w_{\zeta}$ with $|w_{\zeta}|<r$ such that 
$T_\zeta(f)=f(w_{\zeta})$ for all $f\in \cl X$. But, $T_\zeta(f)=(Tf)(\zeta)/\psi(\zeta)$ for all $f\in \cl X$, therefore 
$Tf(\zeta)=\psi(\zeta)f(w_{\zeta}).$ In particular, $(Tz)(\zeta)= \psi(\zeta)w_{\zeta}$ which implies that 
$\phi(\zeta) = w_{\zeta}.$ Hence,  $Tf=\psi.(f\circ \phi)$ for all $f\in \cl X$.
\end{proof}
 
 \begin{remark} We note that Corollary \ref{wcomp0} will hold even when we replace $Hol(\mathbb{D})$with the space $C(\mathbb{D})$ of all complex-valued continuous functions on $\mathbb{D}$ considered with 
 the supremum norm. Of course, in such a case, we can guarantee only the continuity of the functions $\phi$ and $\psi $ in Equation (\ref{wcomp}).
 \end{remark}

\section {Proof of Theorem C: The GKZ theorem for strictly cyclic weighted hardy spaces} \label{cyclic}
In this section, we shall use Theorem A to prove Theorem C. Recall that Theorem C 
is our GKZ theorem for weighted Hardy spaces on 
which the operator of multiplication by the function $z$ is strictly cyclic.  

Let $\{\beta_n\}_{n=0}^\infty$ be a sequence of positive numbers with $\beta_0=1,$ and let $H^2(\beta)$ be the corresponding weighted Hardy space. We assume $M_z$ is a well-defined bounded operator on $H^2(\beta).$ Let $r(M_z)$ denotes its spectral radius and $r_0(M_z)$ equals $\liminf{\beta_n}^{1/n}.$

Furthermore, let $H^\infty(\beta)$ be the set of all formal power series that multiplies $H^2(\beta)$ into itself. Since the constant function $1$ is in $H^2(\beta), \ H^\infty(\beta)$ is a subset of $H^2(\beta).$ In fact, every $\phi\in H^\infty(\beta)$ defines a bounded operator on $H^2(\beta),$ which maps $f\in H^2(\beta)$ to $\phi f;$ we denote it by $M_\phi.$ In addition, if $\phi, \ \psi\in H^\infty(\beta),$ then $\phi \psi$ (the formal Cauchy 
product) is again in $H^\infty(\beta)$ and $M_{\phi \psi}=M_{\phi} M_\psi.$ This allows us to identify 
$H^\infty(\beta)$ as a commutative subalgebra of $\cl B(H^2(\beta)),$ the Banach algebra consisting of all bounded operators on $H^2(\beta).$ Therefore, for $\phi\in H^\infty(\beta),$ if we define $||\phi||_{\infty}:=||M_\phi||,$ then it becomes a commutative unital Banach algebra. Furthermore, the commutant of $M_z$ equals $H^\infty(\beta).$ 

A bounded operator $T$ on a Hilbert space $\cl H$ is said to be {\em strictly cyclic} if there exists $x_0\in \cl H$ such that $\{Rx_0: R\in \cl W[T]\}=\cl H,$ where $\cl W[T]$ is the closure of the algebra of polynomials in $T$ with respect to the weak operator topology. It is proved in \cite[Proposition 31]{Shie} that the operator $M_z$ on $H^2(\beta)$ is strictly cyclic if and only if $H^2(\beta)=H^\infty(\beta)$ (as sets), and in such a case the two norms, namely $||\cdot||_\beta$ and $||\cdot||_\infty,$ are equivalent. Actually, this proposition reveals more about these weighted Hardy spaces. In general, as noted above, we can guarantee the absolute convergence of power in $H^2(\beta)$ only on the disc $\{w: |w|< r_0(M_z)\}$ and $r_0(M_z)\le r(M_z).$ But, if $M_z$ is strictly cyclic, then Proposition 31 from \cite{Shie} establishes that $r(M_z)$ equals $r_0(M_z),$ and each formal power series in $H^2(\beta)$ converges absolutely at every $w$ with $|w|\le r(M_z).$ This in turn implies that the point evaluation at every $w$ with $|w|\le r(M_z)$ defines a bounded linear functional on $H^2(\beta).$  

Before we prove Theorem C, we want to note that one can derive this result using the classical GKZ theorem. Indeed, the assumption $M_z$ is strictly cyclic yields the equality of  
$H^2(\beta)$ and $H^\infty(\beta);$ and then the classical GKZ theorem implies that $F$ must be multiplicative. Further, as recorded in Corollary 1 on Page 94 in \cite{Shie}, when $M_z$ is strictly cyclic, then the multiplicative linear functionals on $H^\infty(\beta)$ are point evaluations with points coming from the closed disc 
$\{w: |w|\le r(M_z)\}$. This establishes the desired conclusion.  

On the contrary, our proof does not rely on the classical GKZ. Instead, it is an elegant application of Theorem A. 

\vspace{.2 cm}

\noindent {\bf Proof of Theorem C.}
 First of all, since $F$ is non-zero on invertible elements and $F(f-F(f))=0,$ therefore $f-F(f)$ cannot be invertible in $H^2(\beta) = H^\infty(\beta)$. This implies that 
$F(f)$ lies in the spectrum of $f$ in $H^\infty(\beta)$ which further implies that $|F(f)|\le ||f||_\infty.$ But $||\cdot||_\infty$ is equivalent to $||\cdot||_\beta,$ therefore $F$ is bounded.   

Now choose a real number $\eta$ such that $\eta>r(M_z)$ and take a scalar $\alpha$ such that 
$|\alpha|\ge \eta.$ This imples that $M_z-\alpha I$ is invertible in $\cl B(H^2(\beta)).$ But, the commutant of $M_z$ equals $H^\infty(\beta),$ therefore $z-\alpha$ is invertible in $H^\infty(\beta)$ which means it is invertible in $H^2(\beta).$ Therefore, $(z-\alpha)^n$ is invertible in $H^2(\beta)$ which implies $F(z-\alpha)^n\ne 0$ for every $n\ge 1.$ Then using Theorem A, there exists a scalar $w$ with $|w|\le \eta$ such that $F(p)=p(w)$ for every polynomial $p.$ In particular, $w=F(z).$ So the point $w$ is independent of the choice of 
$\eta$ and $|w|\le \eta$ for all $\eta>r(M_z).$ Thus, $|w|\le r(M_z).$ Now, $\{z^n/\beta_n\}_{n=0}^\infty$ is an 
orthonormal basis for $H^2(\beta)$ which yields polynomials as a dense set in $H^2(\beta).$ Moreover, point evaluation at $w$ is continuous which establishes that $F(f)=f(w).$    
\qed

\vspace{.2 cm}
 
As promised, we shall now give an example of a weighted Hardy space that is not covered by Theorem B but covered by Theorem C. This means we are looking for a weighted Hardy space $H^2(\beta)$ for which $r_0(M_z)=\liminf{\beta_n}^{1/n}=0$ and on which the operator $M_z$ is strictly cyclic.    

Let $\{w_n\}_{n=0}^\infty$ be a sequence of positive real numbers that is monotonically decreasing and square-summable. Set $\beta_0=1$ and $\beta_n=w_0\cdots w_{n-1}$ for all $n\ge 1.$ Then by \cite[Corollary 4, Page 98]{Shie}, the operator $M_z$ on $H^2(\beta)$ is strictly cyclic. Also, since $w_n$ converges to zero, $r_0(M_z)=r(M_z)=0.$ Actually, the operator $M_z$ is unitarily equivalent to the weighted shift operator $T$ on 
$H^2(\bb D)$ which is given by $T(z^n)=w_nz^{n+1}$ for all $n.$ The operator $T^*$ is famously known as a Donoghue operator. 

\subsection*{Acknowledgements} The first and third authors
thank the Mathematical Sciences Foundation,
Delhi, India for support and facilities needed to complete
the present work.

\end{document}